\documentclass[12pt]{article}

\usepackage{amsmath}
\usepackage{tikz}
\usepackage{amsfonts}
\usepackage{amssymb}
\usepackage{amsthm}
\usepackage{graphicx}

\newcommand{\uw}{\underline{w}}

\newtheorem{theorem}{Theorem}
\newtheorem{lemma}[theorem]{Lemma}

\newtheorem{definition}[theorem]{Definition}
\newtheorem{proposition}[theorem]{Proposition}
\newtheorem{corollary}[theorem]{Corollary}

\newcommand{\pr}[1]{\mathbb{P}\!\left(#1\right)}
\newcommand{\prcond}[2]{\mathbb{P}\!\left(#1\;\middle\vert\;#2\right)}
\newcommand{\expect}[1]{\mathbb{E}\!\left(#1\right)}

\title{The random walk on upper triangular matrices over $\mathbb{Z}/m \mathbb{Z}$}
\author{Evita Nestoridi$^{\ast}$  \and Allan Sly$^{\ast}$ }

\date{}

\begin{document}

\maketitle
\begin{abstract}
We study a natural random walk on the $n \times n$ uni-upper triangular matrices, with entries in $\mathbb{Z}/m \mathbb{Z}$, generated by steps which add or subtract a uniformly random row to the row above. We show that the mixing time of this random walk is $O(m^2n \log n+ n^2 m^{o(1)})$. This answers a question of Stong and of Arias-Castro, Diaconis, and Stanley. 
\end{abstract}

\let\thefootnote\relax
\footnotetext { $^{\ast}$ \textit{Princeton University, USA. E-Mails}: \mbox{exn@princeton.edu}, \mbox{ asly@princeton.edu.}
\hspace{0.85cm}}

\section{Introduction}
Let $n \geq 3$ and $m \geq 2$ be two integers, and let $G_n(m)$ denote the group of $n\times n$ upper triangular matrices with entries in $\mathbb{Z}/m \mathbb{Z}$ and ones along the diagonal, which is also known as the group of uni-upper triangular matrices. We number the rows of each matrix in $G_n(m)$ from top to bottom. We consider the following Markov chain $(X_t)_{t\geq0}$ on $G_{n}(m)$: $X_t$ is derived from $X_{t-1}$ by picking a row $i \in \{2,\ldots, n\}$ uniformly at random and with probability $1/4$ adding it to row $i-1$, with probability $1/4$ subtracting it from row $i-1$, and otherwise staying fixed.

Let $P^t_A(B)$ be the probability that $X_t=B$ given that $X_0=A$. The walk is irreducible with unique stationary measure, the uniform distribution $U$ on $G_n(m)$. Our main result studies the mixing time of the walk, defined as
\begin{equation*}
t_{mix}(\varepsilon):= \inf  \lbrace t\geq 0 : d_n(t) \leq \varepsilon \rbrace,
\end{equation*}
where 
\begin{equation*}
d_n(t) := \max_{A \in G_n(m)} \Vert P^t_{A}-U \Vert_{\text{T.V.}} = \frac{1}{2} \max_{A \in G_n(m)} \big \lbrace  \sum_{B \in G_n(m)} \vert P^t_{A}(B) - U(B) \vert \big \rbrace
\end{equation*}
is the total variation distance of $P^t_A(B)$ from $U$. Our main theorem determines the mixing time of the random walk  $X_t$.

 \begin{theorem}\label{main}
For the random walk $X_t$ on $G_n(m)$, there exist positive constants $ \gamma, \delta, c$, so that for all $n,m \geq 2$ with $m $ prime we have
$$\ d_n(t_{n,m}) \leq e^{- d} ,$$
where $t_{n,m}=t'_{n,m} + d \frac{t'_{n,m}}{\log (n+m)},$ $t'_{n,m}=  \gamma (m^2 n \log n + n^2 e^{\delta \sqrt{ \log m}}) + cnm^2 \log \log n $,  and $d>0$.
 \end{theorem}

For the case where $m$ is not prime, we are missing one of the main tools, namely \cite[Theorem 3]{Hough}. We are still able to prove a similar upper for the mixing time, which is slightly less tight than the one of Theorem \ref{main}.

\begin{theorem}\label{maing}
For the random walk $X_t$ on $G_n(m)$, there exist positive constants $ \gamma, c$ and $\delta$, so that for all $n,m \geq 2$ we have
$$\ d_n(t_{n,m}) \leq e^{- d} ,$$
where $t_{n,m}=t'_{n,m} + d \frac{t'_{n,m}}{\log (n+m)},$ $t'_{n,m}=  \gamma (m^2 n \log n + n^2  e^{\delta (\log m)^{2/3}} ) + cnm^2 \log \log n $ and $d>0$.
\end{theorem}

This process has a long history. The case $n=3$ was first introduced by Zack~\cite{Zack}, and Diaconis and Saloff-Coste proved that order $m^2$ steps are necessary and sufficient for convergence to uniform~\cite{Moderate, Laurent, Har}. This result was later proven again by Bump, Diaconis, Hicks, Miclo and Widom using Fourier analysis~\cite{Hicks}. For $n$ growing, a first upper bound of order $n^7$ was proved by Ellenberg~\cite{Ellenberg}, which was later improved by Stong to $n^3 m^2 \log m $ \cite{Stong}. The case where $m=2$ was treated by Peres and the second author~\cite{PeresSly}, who proved the mixing time is $O(n^2)$. 

Arias-Castro, Diaconis and Stanley \cite{StanleyD} used super-character theory, introduced by Andre \cite{Andre3, Andre, Andre2} and Yan \cite{Yan, Yan2} to bound the spectrum of the random walk. This results in a bound for the mixing time of order $n^4 m^4 \log n $, for the case where $m$ is prime. The first author~\cite{EN} improved their analysis to $n^4 m^2$, which gives the correct order of the mixing time in $m$, but not in $n$.

More recently, some other features of this walk have been studied. For the case where $m$ is a prime, Diaconis and Hough \cite{Hough} studied how many steps an element on the $i$--th diagonal of the matrix needs to mix. Their result, stated in detail in Section \ref{prime}, works best for $m$ fixed. However, we are able to use their bounds for general $n$ and prime $m$ to prove our theorems. For $m=2$, the projection onto the final column of the matrix is itself a well known Markov chain known as the East Model. Ganguly, Lubetzky and Martinelli~\cite{GLM:15} proved that the East model exhibits cutoff and later Ganguly and Martinelli~\cite{GM} extending this to the last $k$ columns of the upper triangular matrix walk, where $k$ is fixed.

Recently, Hermon and Olesker-Taylor \cite{HermonThomas} considered a different question concerning $G_n(m)$. They sample $k$ generators uniformly at random and they prove cutoff for the case where $k$ is growing with $\vert G_n(m) \vert$.

Our strategy is to study how fast the first row mixes and proceed by induction on $n$. We do so by analysing it as a random sum of the second row at random times. It is important to understand the values that the second row takes to understand how well mixed this random sum becomes. This is easier to do in the case where $m$ is a prime, thanks to the work of Diaconis and Hough \cite{Hough}.

\section{Preliminaries}\label{prem}
For the next few sections of the paper we study instead the continuous version of the random walk. For each $ i \in \{ 2,\ldots, n \}$, we consider a rate $1$ Poisson clock, and when the $i$--th clock rings, we either add  or subtract row $i$ to row $i-1$, each event happening with probability $1/4$ or we do nothing. 
Theorems \ref{main} and \ref{maing} can be rewritten for the continuous time random walk as follows.

\begin{theorem}\label{mainc}
For the random walk $X_t$ on $G_n(m)$, there exist positive constants $\alpha ,\beta$, $\gamma$ and $\delta$, so that for all $n,m \geq 2$ with $m $ prime and $\overline t_{n,m}=  \gamma(m^2  \log n + n e^{\delta \sqrt{ \log m}} \log m) + cm^2 \log \log n,$ we have
$$ d_n^{\mathrm{cont}}(\overline t_{n,m}) \leq \beta  e^{-\alpha c},$$
for $c>2$.
 \end{theorem}

 Similarly, we have the following theorem for the continuous time random walk.
\begin{theorem}\label{maingc}
For the random walk $X_t$ on $G_n(m)$, there exist positive constants $\alpha ,\beta, \gamma$ and $\delta$, so that for all $n,m \geq 2$ and  $\overline t_{n,m}=  \gamma (m^2  \log n + n e^{\delta (\log m)^{2/3}}\log m) + cm^2 \log \log n ,$ we have 
$$\ d_n^{\mathrm{cont}}(\overline t_{n,m}) \leq \beta  e^{-\alpha c} ,$$
for $c>2$.
\end{theorem}
Theorems \ref{mainc} and \ref{maingc} say that since each row has its own rate-$1$ clock, the time runs order $n$ steps faster in the continuous time version than in discrete time. 

Also, note that the constants $\alpha ,\beta, \gamma, \delta$ appearing in Theorems \ref{mainc} and \ref{maingc} are not necessarily the same as in Theorems \ref{main} and \ref{maing}, but we prefer to use the same letters for simplicity.
Since from now on we only work with continuous time, we drop the superscript and denote $d_n^{\mathrm{cont}}(t) $ by $d_n(t) $.

We finish this section by showing how to retrieve Theorems \ref{main} and \ref{maing} from Theorems \ref{mainc} and \ref{maingc}.
\begin{proof}[Proof of Theorems \ref{main} and \ref{maing}]
Theorems \ref{mainc} and \ref{maingc} easily imply that there exist positive constants $\alpha ,\beta, \gamma', \delta', c$, so that for all $n,m \geq 2$ with $m $ prime and $t'_{n,m}=  \gamma' (m^2 n \log n + n^2 e^{\delta' \sqrt{ \log m}}) + c nm^2 \log t_{n,m}   \log \log n ,$ we have
 $$\ d_n(t'_{n,m}) \leq \beta' \frac{1}{\sqrt{t'_{n,m}}} ,$$
(see \cite[Theorem 20.3]{Peresbook} and \cite[equation (3.4)]{R}). Using the formula 
\[d(\ell t_{\textup{mix}(\varepsilon)}) \leq (2 \varepsilon)^{\ell},\]
we get
\[d\left(t'_{n,m} + d \frac{t'_{n,m}}{\log (n+m)}\right) \leq \left( \frac{2 \beta'}{\sqrt{t'_{n,m}}} \right)^{ \frac{d}{\log (n+m)}} \leq e^{-d.}\]

\end{proof}

\subsection{The induction lemma}
The goal of this section is to prove an inequality that relates $d_n(t) $ to $d_{n-1}(t)$ thus allowing us to prove Theorems \ref{main} and \ref{maing} inductively.

Let $E(i,j)$ denote the $n \times n$ matrix whose $(i,j) $  entry is one and all other entries are equal to zero. We break $X_t$ in two parts: let $r_t$ be the $n \times n$ matrix that has the same first row as $X_t$ and every other entry zero, and let $Y_t= X_t-r_t$. Let $t_1, t_2 \ldots $ be the times that the second row is selected. Let $N(t)= \max \{ j \geq 0 : t_j \leq t\}$. We have that
\begin{equation}\label{separating}
X_t= Y_t + \sum_{j=1}^{N(t)} a_j E(1,2) Y_{t_j} ,
\end{equation}
where the $a_j \in \{-1,1\}$, with probability $1/4$, or $a_j=0$ with probability $1/2$. 
Equation~\eqref{separating} will allow us to separate the mixing time of the first row from the mixing time of the rest of the matrix.

The main idea is to prove a bound for the $\ell^2$ distance between $r_t$ and the uniform measure on $(\mathbb{Z}/m \mathbb{Z})^{n-1}$, by studying the spectrum of the transition matrix of $r_t$. These eigenvalues are indexed by vectors $y \in (\mathbb{Z}/m \mathbb{Z})^{n}$ whose first coordinate is zero. Let $X_t(i) $ be the $i$--th row of $X_t$. For a nonzero $y \in (\mathbb{Z}/m \mathbb{Z})^{n}$, let $Z_y^t(i)= X_t(i) y $ be the dot product $X_t(i)$ with the column vector $y$. From now on, we will write $y \in (\mathbb{Z}/m \mathbb{Z})^{n-1}$, though we actually mean that $y$ has $n$ coordinates, the first one of which is zero. The $\ell^2$ distance between $r_t$ and its stationary measure at time $t_{n,m}$ is given in terms of $\{Z^s_y(2)  :y \in (\mathbb{Z}/m \mathbb{Z})^{n-1} , s \in \{ t_1, \ldots, N(t_{n,m})\} \}$. We prove that most of these values are conveniently large. For $a,b \in \mathbb{Z}/m \mathbb{Z}$, we use the notation $|a| >b$ to mean that $a \in \{b+1, \ldots, m-b-1\} $ .

Let $P^t$ be the Poisson point process counting how many times the clock of the second row has rang by time $t$. 
Let $A^t_{y,x}= \int_0^t  1_{\lbrace \vert Z_y^s(2)   \vert >   x\rbrace } dP^s$ .
Let $\langle y \rangle$ denote the vector space generated by $y$. Let $I \in \{2,\ldots, n \}$ and let $P_2=\langle e_1, e_2 \rangle \setminus \langle e_1 \rangle $, $Q_I= \langle e_1, \ldots e_{I-1} \rangle \setminus \langle e_1, e_2 \rangle $ and $W_I\in (\mathbb{Z}/m \mathbb{Z})^{n-1} \setminus  \langle e_1, \ldots e_{I-1} \rangle$.  
\begin{definition}\label{event}
Let $x_y$ and $ A_{y,x_y}=A(m)$ be constants that will be determined later, and let $E_{t,I,x_y}$ be the event that $A_{y,x_y}^t \geq A_{y,x_y}t$ for every $y \in W_I$, $A_{y,x_y}^t \geq A_{y,x_y}t$ for every $y \in Q_I$ and $A_{y,x_y}^t \geq A_{y,x_y}t$ for every $y \in P_2$. 
\end{definition}
The following lemma will help us prove Theorems \ref{mainc} and \ref{maingc} inductively.  Let $\mathcal{F}_t$ be the $\sigma-$algebra generated by the all of updates except the random signs used when adding/subtracting the second row to the first up to time $t$. In particular, $\mathcal{F}_t$ contains all the information from rows 2 to $n$ as well as the times at which the clock assigned to row 2 rings. Let $q_{t}$ be the conditional distribution of $r_t$ at time $t$ given $\mathcal{F}_t$.


\begin{lemma}\label{first} We have
$$d_{n}(t) \leq d_{n-1}(t) + \expect{ \Vert q_{t}-u \Vert_{T.V.} },$$
where $u$ is the uniform measure on $(\mathbb{Z}/ m \mathbb{Z})^{n-1}$.
\end{lemma}

\begin{proof}

Let $X$ be a uniformly random element of $G_{n}(m)$. We can couple $Y_t$ and the $n-1$ last rows of $X$ except with probability $d_{n-1}(t)$. This coupling moreover can be made $\mathcal{F}_t$ measurable.  Conditional on $\mathcal{F}_t$ we can then couple the first row except with probability $\Vert q_{t}-u \Vert_{T.V.}$.  The lemma then follows by averaging.
\end{proof}


We use Lemma \ref{first} to prove Theorems \ref{mainc} and \ref{maingc} by induction. In particular, we prove the following proposition.

\begin{proposition}\label{q}
Let $t=\overline t_{n,m}$. There exist positive constants $a, b,\gamma$ such that for any $n$ we can find $x,w$ and $I$ so that
$$ \Vert q_{\overline t_{n,m}}-u \Vert_{T.V.} \cdot 1_{E_{\overline t_{n,m},I,x_y}}  \leq a n^{-1}(\log n)^{-c},$$
where $c>0$ is the constant from Theorems \ref{main} and \ref{maing}.
Furthermore, for the same $x, y, I$,
$$ \pr{E_{\overline t_{n,m},I,x_y} ^c} \leq  
 b e^{-c} \left(\frac{1}{m^{\gamma n}} +  \frac{1}{n^{\gamma m}} \right).
$$
\end{proposition}
We now show how to use Proposition \ref{q} to prove Theorems \ref{mainc} and \ref{maingc}.
\begin{proof}[Proof of Theorems \ref{mainc} and \ref{maingc}.]
The proof for the case $n=3$ can be found in \cite[Theorem 1.1]{Moderate}. Combining Lemma \ref{first} and Proposition \ref{q} with induction, we get that 
\begin{align*}
d_n(t_{n,m}) & \leq \overline{B}  e^{-dc}+ \frac{(\log 2)^{1-c}}{c-1} + b e^{-c}\sum_{i=4}^n \left(\frac{1}{m^{\gamma i}} +  \frac{1}{i^{\gamma m}} \right)  , 
\end{align*}
where the term $\overline{B}  e^{-dc}$ comes from the case $n=3$. Therefore,
\begin{align*}
d_n(t_{n,m}) & \leq \overline{B}  e^{-dc} +\frac{(\log 2)^{1-c}}{c-1} + \tilde{b} e^{-fc},\\
& \leq \beta e^{- \alpha c}
\end{align*}
for $c>2$, which completes the proof of Theorems \ref{mainc} and \ref{maingc}.
\end{proof}

\subsection{The $\ell^2$ bound}
The goal of this section is to establish an inequality which will be used to bound $\Vert q_{t}-u \Vert_{T.V.}$. Let $N(t)$ be the number of times that the second clock has rung by time $t$. 
For $k \in \mathbb{N}$ and $\uw=(w_1, \ldots, w_k)$ with $w_i \in ( \mathbb{Z}/m \mathbb{Z})^{n-1}$ let $G_{k,\uw}$ be the event that $N(t)=k$ and that the second row $X_t(2)$ is equal to $w_j$ at the time of the $j$-th ring for $j=1,\ldots,k$. 

Let $q_{\substack{k,\uw}}$ be the distribution of $r_t$, conditional on $G_{k,\uw}$.
Then
\begin{align}
\Vert q_{t}-u \Vert_{T.V.} =
\label{l} \sum_{\substack{k, \uw}}\Vert q_{\substack{k,\uw}}-u \Vert_{T.V.} \prcond{G_{k,\uw}}{\mathcal{F}_t}.
\end{align}
Each $q_{\substack{k,\uw}}$ has the same orthonormal eigenbasis as the simple random walk on $ \mathbb{Z}/m \mathbb{Z}$, despite the fact that at each step we are adding/subtracting a different quantity. Let $(y \cdot w_s)$ be the dot product of $y$ and $ w_s$.The corresponding eigenvalues are $e^{-2(k- \sum_{s=1}^k \lambda_{\substack{y, w_s}})},$
where the $ \lambda_{\substack{y, w_s} }=  \cos \frac{2 \pi (y \cdot w_s) }{m}$ are the eigenvalues of the discrete time Markov chain on $ ( \mathbb{Z}/m \mathbb{Z})^{n-1}$ that adds or subtracts $w_s$ to the current state with probability $1/2$ (see \cite[Example 2.1]{Threads} for a reference).
Then we use the classical $\ell^2 $ bound, 
\begin{align}\label{y}
&4 \Vert q_{\substack{k, \uw}}-u \Vert_{T.V.}^2 \leq \sum_{  y  \in (\mathbb{Z}/m \mathbb{Z})^{n-1} \setminus \{ \textbf{0} \} } e^{-2(k- \sum_{s=1}^k \lambda_{\substack{y, w_s}})}.
\end{align}

To continue with bounding \eqref{y}, we will need the following technical lemma.
\begin{lemma}\label{integral}
We have 
$$\sum_{j =1}^{m-1} 
e^{-2 x (1-   \cos \frac{2 \pi j }{m} )} \leq m e^{-2x} + \frac{\sqrt{3} m}{ 2\sqrt{2 \pi  x}} ,$$
where $x>0$.
\end{lemma}
\begin{proof}
Using properties of the cosine, we get
\begin{align}
 \sum_{j =1}^{m-1} 
e^{-2  x (1-   \cos \frac{2 \pi j }{m} )} 
& \leq  2 \sum_{j =1}^{m/2} 
e^{-2   x (1-   \cos \frac{2 \pi j }{m} )}\cr
& \label{spl} \leq m e^{-2x} +  2 \sum_{j =1}^{m/4} e^{-2 x (1-   \cos \frac{2 \pi j }{m} )} ,
\end{align}
where for the first term in \eqref{spl} we bound the negative cosines. Using the inequality $\cos x \leq 1 - \frac{x^2}{2} + \frac{x^4}{24}$, we get that
\begin{align}
\eqref{spl}& \label{thr} \leq m e^{-2x} + 2 \sum_{j=1}^{m/4} e^{- \frac{8j^2 \pi^2}{3 m^2}  x }\\
&\label{i} \leq m e^{-2x} + 2  \int_{0}^{\infty }  e^{- \frac{8w^2 \pi^2x  }{3m^2} }dw.
\end{align}
Using the substitution $v=\frac{4 \pi \sqrt{x}  }{\sqrt{3}m}w,$ we get that
\begin{align*}
\eqref{i}&\leq m e^{-2x} + \frac{ \sqrt{3}m}{ 2\pi  \sqrt{  x}} \int_{0}^{\infty }  e^{-v^2/2}dv \cr
& \leq m e^{-2x} + \frac{\sqrt{3} m}{ 2\sqrt{2 \pi  x}} .
\end{align*}
\end{proof}

\subsection{Coupling with Exponentials}\label{good}
Recall that $Z^t_y(i):= X_t (i)y,$ for $i=1,\ldots, n$.  In this section, we study the time intervals during which $Z^t_y(i) \neq 0$. We want to understand for how long $Z^t_y(i) $ remains equal to zero and for how long it does not. This way we can argue that a good number of $ \lambda_{\substack{y, w_s} }$ is not equal to $1$, thus contributing to the right hand side of \eqref{y}.

Let $y \in (\mathbb{Z}/m \mathbb{Z})^{n-1} \setminus \langle e_1 \rangle,$ where $ \langle e_1 \rangle$ denotes the subspace of $ (\mathbb{Z}/m \mathbb{Z})^{n-1}$ generated by the vector $e_1=(1,0,\ldots ,0)$. 
We start by proving that there is a good chance that $Z^t_y(i)$ will be non-zero after order $ n  $ steps.
\begin{lemma}\label{nzero}
Let $T_i$ denote the first time that $Z^t_y(i)$ is non-zero. We have that
$$\pr{T_i> 24(n   +c) } \leq  e^{-c},$$
for $c>0$.
\end{lemma}
\begin{proof}
To study the tails of $T_i$, we will follow the position $P_t$ of the first non-zero entry in the column dynamics $Z^t_y$. We note that $P_0=n$. We are going to couple $P_t$ with the following random walk $S_t$ on $\mathbb{Z}$, starting at zero. 
Consider the column dynamics, $Z^t_y =(Z^t_y(i))_{i=2}^n$. Let $Z_t$ be the first entry of the column $Z^t_y$ that is not divisible by $m$, when read from top to bottom. The worst starting point is at the second coordinate, in which case we want to study the first time $\xi$ that $Z$ returns at $2$. 

Whenever $Z_t$ is added to or subtracted from the coordinate directly above, then $Z_t$ moves up by one (at rate one). Since $m$ is odd, $Z_t$ can move down by one coordinate by at most one move between adding or subtracting (so at most at rate $1/4$). Therefore, we can couple $Z_t$ with a biased random walk $S_t$ on  $\mathbb{Z}$.

When at $x$, $S_t$ moves to $x+1$ according to a rate $1/2$ Poisson clock or to $x-1$ according to a rate $1/4$ Poisson  clock. A Chernoff bound gives that
\begin{equation}\label{t}
\pr{T_i> t } \leq  \pr{S_t <n-1}= \pr{e^{- \lambda  S_t} > e^{- \lambda (n-1)}} \leq e^{\lambda (n-1)} \expect{e^{-\lambda S_t}}
\end{equation}
We have that $S_t= M_t-N_t,$ where $M_t$ is a Poisson($t/2$) random variable and $N_t$ is a Poisson($t/4$) random variable. Therefore
\begin{equation}\label{poi}
\expect{e^{-\lambda M_t}}= e^{\frac{t}{2}(e^{-\lambda}-1)} \textup{ and }\expect{e^{\lambda N_t}}=e^{\frac{t}{4}(e^{\lambda}-1)}. \end{equation}
Using the fact that $M_t$ and $N_t$ are independent, we get
\begin{equation}\label{mgf}
 \eqref{t} \leq e^{\lambda n} e^{-\frac{3}{4}t+\frac{t}{2} e^{- \lambda } + \frac{t}{4}e^{\lambda}}.
\end{equation}
Setting $\lambda=  \log \sqrt{2}$, we get that
$$\pr{T_i> t } \leq  2^{n} e^{-t /12}.$$
Setting $t= 24( n+c)$ we get the desired result.
\end{proof}
We want to study for how long $Z^t_y(2)$ can remain divisible by $m$, which means that the corresponding eigenvalues in the right hand side of \eqref{y} will be equal to one. 
\begin{definition}
Let $\ell_1$ be a time such that $Z^{\ell_1}_y(2)=0$ and $Z^{\ell_1^-}_y(2)\neq 0$. Let $\ell_2= \inf \{t>\ell_1: Z^{t}_y(2)\neq 0 \}$. We will call $[\ell_1, \ell_2] $ a $(y,I)$--zero interval.
\end{definition}

\begin{lemma}\label{nstar}
 Let $m>2$ be an odd integer. Let $[\ell_1, \ell_2]$ be a $(y,i)$--zero interval. Then,
$$\prcond{\ell_2-\ell_1>12k }{\mathcal{F}_t} \leq  e^{-k},$$
where $k>0$.
\end{lemma}

\begin{proof}

Let $S_t$ be the random walk on $\mathbb{Z}$, which starts at $0$, and moves by $+1$ according to a rate $1/2$ Poisson clock and by $-1$ according to a rate $1/4$ Poisson clock.
Note that $S_t= M_t-N_t,$ where $M_t$ is a Poisson($t/2$) random variable and $N_t$ is a Poisson($t/4$) random variable, and thus 
\begin{equation}\label{poi2}
\expect{e^{-\theta M_t}}= e^{\frac{t}{2}(e^{-\theta}-1)} \textup{ and }\expect{e^{\theta N_t}}=e^{\frac{t}{4}(e^{\theta}-1)}. \end{equation}
Using the fact that $M_t$ and $N_t$ are independent, a Chernoff bound gives
        $$\pr{\xi>x} \leq \pr{S_x <0} \leq \pr{ e^{-\theta S_x} \geq 1 } \leq \expect{e^{-\theta S_x}} \leq     e^{-\frac{3x}{4} +\frac{1}{2}x e^{-\theta}+ \frac{1}{4}x e^{\theta}}.$$
Setting $\theta= \log \sqrt{2}$, we have that 
$$\pr{\xi>x}   \leq e^{- \frac{x}{12}} .$$
Therefore, if $x= 12k $,
\[\prcond{\ell_2-\ell_1>12k }{\mathcal{F}_t} \leq  e^{-k}. \qedhere \]
\end{proof}

More importantly, we will study the length of the intervals during which $Z^t_y(2)\neq 0$. This will help us understand the non-trivial terms on the right hand side of \eqref{y}.

\begin{definition}
Let $\ell_3$ be a time such that $Z^{\ell_3}_y(i)\neq0$ and $Z^{\ell_3^-}_y(i)= 0$. Let $\ell_4= \inf \{t>\ell_3: Z^{t}_y(i)=0 \}$. We will call  $[\ell_3, \ell_4] $ a $(y,i)$--non-zero interval.
\end{definition}

\begin{lemma}\label{star} 
Let $m>2$ be an integer. Let $[\ell_3, \ell_4] $ be  a $(y,i)$--non-zero interval. Then,
$$\prcond{\ell_4-\ell_3 \geq k }{\mathcal{F}_t} \geq  e^{-k},$$
where $k>0$.
\end{lemma}
\begin{proof}
The $i$--th entry of $Z_y$ can turn from zero to non-zero if the $i+1$ clock rings and $Z_y(i+1) $ has the appropriate value. Therefore, we can couple $\ell_4-\ell_3$ with the time it takes for the clock of the $i+1$ row to ring and the statement follows. 
\end{proof}

We are now going to put all this information together to prove that during any interval, $ Z^t_y(2)$ is non-zero for a constant fraction of the time. We break up the interval $[0,\overline t_{n,m}] $ in intervals $[t_j,  t_{j+1}]$ of length $L $, so that $t_j=jL$. Let $j \in \{ 1,\ldots,n\}$ and let $g=99$. 
\begin{definition}\label{g}
Let $i \in \{1,\ldots, n-1\} $. An interval $[t_j,  t_{j+1}] $ is called $(y,i)$--good if $ Z^t_y(i) \neq 0 \mod m $  for at least $1/g$ of $[t_j, t_{j+1}]$.
Let $D_{y}^i$ be the set of all $(y,i)$--good intervals by time $\overline t_{n,m}$. Let $M_{y,i}$ be the number of $(y,i)$--good intervals that have occurred by time $\overline t_{n,m}$.
\end{definition}

The following lemma follows from a simple Chernoff bound and will help us study how likely it is for a given interval to be $(y,i)$--good.
\begin{lemma}\label{expon}
Let $E_1, \ldots, E_k$ be independent, exponential random variables with mean one. We have that
\begin{enumerate}
\item[(a)] $\pr{\sum_{i=1}^k E_i > 2k} \leq \left( \frac{2}{e} \right)^{k}.$
\item[(b)] $\pr{\sum_{i=1}^k E_i < \frac{k}{2}} \leq \left(\frac{6}{7}\right)^k$.
\end{enumerate}
\end{lemma}
%
The following lemma says that a constant fraction of intervals are $(y,i)$--good.
\begin{lemma}\label{prob}
Let $i \in \{3, \ldots, n \rbrace$. At time $\overline t_{n,m}$ we have that
$$\pr{M_{y,i} \leq \frac{\overline{t}_{n,m}}{100L}} \leq e^{-d_1 \overline{t}_{n,m}},$$
for a suitable constant $d_1$.
\end{lemma}
\begin{proof}
Consider the $(y,i)$--non-zero intervals $A_b \subset [0,\overline t_{n,m}]$  and the $(y,i)$--zero intervals $B_k \subset [0,\overline t_{n,m}]$. Let $\vert A_b \vert$ be the length of $A_b$. Let
$$W_y= \sum_b \vert A_b \vert $$ 
be the total time that $Z^s_y(i)$ is not equal to zero. 
Notice that

\begin{align}
W_y& =   \sum_b \vert A_b \vert \leq L \sum_j   1_{ \{ [t_j, t_{j+1}] \in D^i_y\}} +\frac{L}{g}   \sum_j    1_{\{[t_j, t_{j+1}] \notin D^i_y\}}\cr
&=  L M_{y,i} + \frac{L}{g} \left( \frac{\overline t_{n,m}}{L} -  M_{y,i}\right)\cr
& \label{mm} = L\left(1- \frac{1}{g}\right) M_{y,i} + \frac{1}{g} \overline t_{n,m}.
\end{align}
For $m$ odd, equation \eqref{mm} gives that 
\begin{align}
\label{nine1}& \bigg \lbrace M_{y,i} \leq  \frac{\overline t_{n,m}}{100L} \bigg \rbrace \subset \lbrace W_y\leq x \overline t_{n,m}\rbrace ,
\end{align}
where $x= \frac{1}{100}+ \frac{99}{100g}= \frac{1}{50}.$

Lemmas \ref{nstar} and \ref{star} say that we can couple each $\vert A_b \vert$ and $\vert B_k\vert /12$ with exponential random variables with mean $1$. Let $ r =\frac{2}{49} \overline t_{n,m}$ and $c= \frac{1}{48} \overline t_{n,m}$. 

Either $ \cup_{j=1}^r A_j \subset [0, \overline t_{n,m}]$ or $(\cup_{j=0}^r B_j)^c  $
contains all $(y,i)$--non-zero intervals that $[0, t_{n,m}]$ contains. This is summarized in the following equation
$$W_y \geq \min \big \{\sum_{b=1}^r \vert A_b \vert, \overline t_{n,m} -B_0- \sum_{k=1}^r \vert B_k \vert \big \}.$$ 
Therefore,
\begin{align}
\pr{W_y\leq x \overline t_{n,m}}   &\leq  \pr{\sum_{b=1}^r \vert A_b \vert \leq x \overline t_{n,m} } +\pr{\vert B_0 \vert \geq 24(n+c)}\cr
& \label{r}  + \pr{\sum_{k=1}^r \vert B_k \vert \geq \left( 1-x \right)\overline t_{n,m} -24(n+c) }.
\end{align}
We choose $\gamma$ from Theorems \ref{mainc} and \ref{maingc} approprietely so that $\left( 1-x \right)\overline t_{n,m} -24(n+c) \geq 2r$. Then,
\begin{align}
\eqref{r} & \leq  \pr{\sum_{b=1}^r \vert A_b \vert \leq r/2}+ e^{-c} + \pr{\sum_{k=1}^r \bigg \vert \frac{B_k}{12} \bigg \vert \geq 2r}.
\end{align}
Lemmas \ref{nstar} and \ref{star} say that $\vert A_b \vert$ and $\vert B_k \vert$ are stochastically dominated by appropriate exponentials above and below respectively. Lemma \ref{expon} gives 
\begin{align}
\eqref{r} &\label{e} \leq  \left(\frac{6}{7} \right)^{r} +e^{-c} +\left( \frac{2}{e} \right)^{ r} .
\end{align}
Putting \eqref{nine1} and \eqref{e} together, we have
 \begin{align}
 \label{nine}&\pr{M_{y,i} \leq  \frac{\overline t_{n,m}}{100L}} \leq   \left(\frac{6}{7} \right)^{r}+ e^{-c} + \left( \frac{2}{e} \right)^{ r}  .
 \end{align}
Using the definition of $r, c$ and \eqref{nine} we get the desired result.

For the case where $m$ is even, project all values over $\mathbb{Z}/2 \mathbb{Z}$. Equation 2.2 of \cite{PeresSly} says that for every $\varepsilon>0$, we have that
$$\pr{\big \vert W_{\overline t_{n,m}} - \frac{\overline t_{n,m}}{2} \big \vert \geq \varepsilon \overline t_{n,m}} \leq 2^{n+1} e^{- \frac{\overline t_{n,m} \varepsilon^2 \lambda}{12}},$$
where $\lambda$ is a positive constant not depending on $n,m$.
Setting $\varepsilon= \frac{12}{25}$ we get
\begin{align*}
\pr{M_{y,i}\leq \frac{\overline t_{n,m}}{100L}} \leq  &\pr{W_y\leq x \overline t_{n,m}} \leq 2^{n+1} e^{- \frac{12  \lambda \overline t_{n,m} }{625}}.
\end{align*}
\end{proof}
Finally, we need a lemma that guarantees that $ Z_y(2)   $ is sufficiently big for a constant fraction of $[0, \overline t_{n,m}]$ with high probability.

Let $I \in [n]$.
Recall that $P^t$ is a Poisson point process according to which the clock of the second row rings and recall that $A^t_{y,x}= \int_0^t   1_{\lbrace \vert Z_y^s(2)   \vert >   x\rbrace } dP^s$. 
Let $\mathcal{F}_{j}$ be the $\sigma$--algebra generated by all the clock rings before time $t_j$. Let $\mathcal{F}_{j}^-$ be the $\sigma$--algebra generated by all the clock rings ${I+1, \ldots n}$ before time $t_j$ and set $\mathcal{F}_{j}^*= \sigma(\mathcal{F}_{j}, \mathcal{F}_{j+2}^-)$.

In later sections, we set $I= \min \{ \lfloor \sqrt{\log m} \rfloor , n-1 \}$ if $m$ is prime, otherwise we set $I=\min \{ \lfloor \sqrt[3]{\log m}\rfloor , n-1\}$.  We consider different regimes for $y$ and study the corresponding values of $Z_y^{t}(2)$. Namely, we consider 
\begin{enumerate}
\item $y \in W_I:= \left( \mathbb{Z}/ m \mathbb{Z}\right)^{n-1} \setminus \langle e_1, \ldots, e_{I-1} \rangle$,
\item $y\in Q_I:= \langle e_1, \ldots, e_{I-1} \rangle \setminus \langle e_1 , e_2\rangle$,
\item $y \in P_2:=\langle e_1 , e_2\rangle$.
\end{enumerate}
The goal is to prove that $\vert Z^t_y(2)  \vert$ is big very often. To quantify how big we need $\vert Z^t_y(2)  \vert$ to be, we define the following quantity $x_y$. We also specify the length of the intervals for each $y$. More precisely, we define
\begin{equation}\label{definition}
x_y= 
\begin{cases}
m/8, & \textup{ if $m$ is prime and } y \in W_I  \\
m e^{-k (\log m)^{2/3}}, & \textup{ if $m$ is not prime and } y \in W_I  \\
m/K, & \textup{ if } y \in Q_I \\
\sqrt{ \log m}/2, & \textit{ if } y \in P_2
\end{cases}
\end{equation}
and 
\[
L_y= 
\begin{cases}
d_2 m^{4/I}, & \textup{ if $m$ is prime and } y \in W_I  \\
2f m^{2/J}/ \log h, & \textup{ if $m$ is not prime and } y \in W_I  \\
Fm, & \textup{ if } y \in Q_I \\
\delta_1  \lfloor \log m \rfloor , & \textit{ if } y \in P_2,
\end{cases}
\]
where $k,K, d_2, f, F, \delta_1$ are suitable constants, $J= \min \{ \lfloor (\log m)^{1/3}, n-2\}$ and $h=20(J+1)$.

Let $B_{j,y}$ denote the event that $\vert Z^s_y(2)  \vert >  x_y$ for at least a proportion $A^{-1}$ of 
$ [t_{j+1}, t_{j+2}]$. Let $\mathcal{G}^y_j$ be the event that $[t_j,t_{j+1}] \in D_y^I$. The next sections are focusing on proving the following lemma.
\begin{lemma}\label{r115}
 For $y \notin \langle e_1 \rangle$,  there is a uniformly bounded, positive constant $\zeta$ such that $\prcond{B_{j,y}}{ \mathcal{F}_{j}^*} \geq \zeta 1_{\mathcal{G}^y_j}$ for all $j$. 
\end{lemma}
Lemma \ref{r115} is crucial to proving the next lemma, which is the main ingredient for proving the second part of Proposition \ref{q}.
\begin{lemma}\label{r15}
At time $\overline t_{n,m}$, we have
\begin{align*}
\pr{A^{\overline{t}_{n,m}}_{y,x}
< (400A)^{-1}\zeta \overline{t}_{n,m}} \leq  \left( \frac{2}{e} \right)^{\frac{\zeta \overline{t}_{n,m}}{400A}}+  e^{-\overline{t}_{n,m}/2}+e^{- d_1 \overline{t}_{n,m}}+ 2e^{-\frac{\zeta^2 \overline{t}_{n,m}}{400^2L}},
\end{align*}
where $d_1$ is as in Lemma \ref{prob}, $\zeta$ is as in Lemma \ref{r115} and $L$ is the length of the intervals.
\end{lemma}
\begin{proof}
The event $\{\int_0^{\overline{t}_{n,m}}   1_{\lbrace \vert Z_y^s(2)   \vert >   x\rbrace } ds
\geq (200A)^{-1} \zeta \overline{t}_{n,m}\}$ is satisfied if at least $\zeta/200$ of the $B_{j,y}$ are satisfied. For this reason we want to estimate $\sum_j I_{B_{j,y}}$. We write
\begin{align}\label{13}
\sum_j I_{B_{j,y}} & = \sum_j \prcond{B_{j,y}}{\mathcal{F}^*_{j}} + \sum_{j} \left( I_{B_{j,y}}- \prcond{B_{j,y}}{\mathcal{F}_{j}^*}\right).
\end{align}
We set $U_{\ell}=\sum_{j\leq \ell} \left( I_{B_{j,y}}- \prcond{B_{j,y}}{\mathcal{F}_{j}^*}\right)$ and $S_1=\{U_{\frac{\overline{t}_{n,m}}{L}} \geq -(200L)^{-1} \zeta \overline{t}_{n,m}\}$. Let $S_2$ be the event that $\{\sum_j \prcond{B_{j,y}}{\mathcal{F}^*_{j}} \geq (100L)^{-1}\zeta \overline{t}_{n,m} \}$. Equation \eqref{13} gives
\begin{align}
\pr{ \sum_j I_{B_{j,y}} \geq (200L)^{-1}\zeta \overline{t}_{n,m}} & \label{MM12} \geq \pr{ S_1 \cap S_2}.
\end{align}
We note that $X_t= \sum_{j=1}^{t-1} \left( I_{B_{2j,y}}- \prcond{B_{2j,y}}{\mathcal{F}_{2j}^*}\right) $ is a martingale with respect to $\mathcal{F}_{2t}^*$. Using the Azuma-Hoeffding inequality we have 
\begin{align*}
 \pr{ X_{\frac{\overline{t}_{n,m}}{2L}} \geq -(400L)^{-1}\zeta t_{n,m} } &\geq 1- e^{-\frac{\zeta^2 \overline{t}_{n,m}}{400^2L}} .
\end{align*}
Similarly, using a time shifting argument, we can get
\begin{align*}
 \pr{ \overline{X}_{\frac{\overline{t}_{n,m}}{2L}} \geq -(400L)^{-1} \zeta t_{n,m} } &\geq 1- e^{-\frac{\zeta^2 \overline{t}_{n,m}}{400^2L}} ,
\end{align*}
where $\overline{X}_t= U_{2t}-X_t$. Therefore,
\begin{align}
\label{M12} \pr{ U_{\frac{t_{n,m}}{L}} \geq -(200L)^{-1}\zeta \overline{t}_{n,m} } &\geq 1- 2e^{-\frac{\zeta^2 t_{n,m}}{400^2L}} .
\end{align}
To continue bounding the right hand side of \eqref{MM12}, we also want to bound $\pr{ \sum_j \prcond{B_{j,y}}{\mathcal{F}_{j}^*}}$. Recalling Lemma \ref{r115}, we have that  $\prcond{B_{j,y}}{ F_{j}^*} \geq \zeta 1_{\mathcal{G}^y_j}$. Thus,
\begin{align}
   \label{y12}  \pr{ \sum_j \prcond{B_{j,y}}{\mathcal{F}^*_{j}} \geq (100L)^{-1} \zeta \overline{t}_{n,m} }\geq  \pr{ M_{y,I} \geq \frac{t_{n,m}}{100  L}} \geq 1-e^{- d_1 t_{n,m}},
\end{align}
by Lemma \ref{prob}.
A union bound and equations \eqref{13}, \eqref{MM12}, \eqref{M12} and \eqref{y12} give
\begin{equation}\label{tnm}
\pr{\int_0^{\overline{t}_{n,m}}   1_{\lbrace \vert Z_y^s(2)   \vert >   x\rbrace } ds
< (200A)^{-1}\zeta t_{n,m}} \leq  e^{- d_1 \overline{t}_{n,m}}+ 2e^{-\frac{\zeta^2 t_{n,m}}{400^2L}}. 
\end{equation}

We finish the proof using a Poissonization argument. Conditioning on $\Lambda_{\overline{t}_{n,m}}:=\int_0^{t_{n,m}}   1_{\lbrace \vert Z_y^s(2)   \vert >   x\rbrace } ds
=s $, we have that $A^t_{y,x}$ is a Poisson random variable with mean $s$.
\begin{align*}
\pr{A^t_{y,x}
< (400A)^{-1} \zeta \overline{t}_{n,m}} &\leq \prcond{A^{t_{n,m}}_{y,x}
< (400A)^{-1}\zeta \overline{t}_{n,m}}{\Lambda_{t_{n,m}}
\geq (200A)^{-1} \zeta \overline{t}_{n,m}} \\
& \quad+\pr{\Lambda_{t_{n,m}}
< (200A)^{-1} \zeta \overline{t}_{n,m}}.
\end{align*}
Using the tails of a Poisson random variable and \eqref{tnm}, we get
\begin{align*}
\pr{A^t_{y,x}
< (400A)^{-1} \zeta t_{n,m}} 
& \leq \left( \frac{2}{e} \right)^{\frac{\zeta \overline{t}_{n,m}}{400A}}+  e^{- d_1 \overline{t}_{n,m}}+  2e^{-\frac{\zeta^2 \overline{t}_{n,m}}{400^2L}},
\end{align*}
which finishes the proof.
\end{proof}

\subsection{Reducing the walk to a smaller dimension}
Let $I \in  \lbrace 2,\ldots,n \rbrace$. In this section, we develop the necessary tools to study the walk if we only focus on the top $I$ coordinates of $Z_y^t$.

Let $  s_1 <s_2< \ldots \leq \overline{t}_{n,m}$ denote the times when the $I$--th clock rings and $Z_y^{s_j}(I) \neq 0$. Let $z_1< z_2<\ldots \leq \overline{t}_{n,m}$ denote the times when a clock other than the $I$--th one rings and let $\{ W_j \}_{j=1}^{\overline{t}_{n,m}}$ be the corresponding operation matrix applied. This means that if at time $z_j$, row $i$ is added to row $i-1$, then
\[W_j= I_n+E(i-1,i),\]
where $E(i-1,i)$ denote the matrix whose $(i-1,i) $  entry is one and all other entries are equal to zero and $I_n$ is the identity matrix. 

Let $L(t)= \max \lbrace j \geq 0 :z_{j} \leq t \rbrace$. For $0\leq t \leq \overline{t}_{n,m}$, we define the backwards process by $Y_0=I_n$ and 
\begin{equation}
Y_t= \prod_{j=0}^{L(\overline{t}_{n,m} ) - L(t- \overline{t}_{n,m} ) -1}W_{L(\overline{t}_{n,m})-j} =W_{L(\overline{t}_{n,m})}W_{L(\overline{t}_{n,m})-1}\ldots  W_{L(\overline{t}_{n,m}-t)+1}
\end{equation}
and for $0\leq  t'< t \leq \overline{t}_{n,m}$ we let 
\begin{equation}
Y_{t',t}= Y_{\overline{t}_{n,m}-t}^{-1} Y_{\overline{t}_{n,m}-t'}=W_{L(t)}\ldots  W_{L(t')+1}.
\end{equation}
Notice that the entries of $Y_t$ and $Y_{t,t'}$ that fall on the $[1,I-1] \times [I,n]$ box are equal to zero and that $Y_t$ is a Markov chain on the columns of a matrix in $G_n(m)$.

The next lemma explains the connection between $Z_y^{s_{\ell}}$ and the $Y_{t',t}$'s. 
\begin{lemma}\label{vector}
We have
\begin{align*}
Z_y^{s_{\ell}} & = Y_{0, s_{\ell} }  Z_y^{0} + \sum_{k=1}^{\ell-1} a_k  Y_{s_k, s_{\ell} } E(I-1,I) Y_{ 0, s_{k}} Z_y^{0}  + a_{\ell} E(I-1,I)  Y_{ 0,s_{\ell} } Z_y^{0},
\end{align*}
where $a_k \in \{ \pm 1, 0 \}$ are the random signs corresponding to the $k$--th time the $I$--th clock rings. 
\end{lemma}
\begin{proof}
We prove the statement by induction. For $\ell=0$ both sides are equal to $Z_y^{0}$. By the definition of $s_{\ell+1}$ we have 
\begin{align}
\label{defi}Z^{s_{\ell +1}}_y & = \left( I_n + a_{\ell+1} E(I-1,I) \right) Y_{   s_{\ell}, {s_{\ell +1}}} Z^{s_{\ell }}_y .
\end{align}
By the induction hypothesis we have
\begin{align}
\eqref{defi} & = \left( I_n + a_{\ell +1} E(I-1,I) \right)   Y_{ s_{\ell}, {s_{\ell +1}}}  \biggl( Y_{0, s_{\ell} }  Z_y^{0}+ a_{\ell} E(I-1,I)  Y_{0, s_{\ell} } Z_y^{0}   \cr
\label{special}&\quad\quad + \sum_{k=1}^{\ell} a_k  Y_{ s_k,s_{\ell}}  E(I-1,I) Y_{0, s_k} Z_y^{0} \biggr).
\end{align}
Using the facts that $ E(I-1,I) E(I-1,I)= 0$ and $ E(I-1,I)Y E(I-1,I)= 0$ for every $Y \in G$ whose $[I-1] \times [ I,n]$ entries are zero, we get
\begin{align*}
\eqref{special}& = Y_{0, s_{\ell+1} } Z_y^{0}+ a_{\ell +1}   E(I-1,I) Y_{0, s_{\ell+1} } Z_y^{0} +\sum_{k=1}^{\ell-1} a_k  Y_{s_k, s_{\ell+1}} E(I-1,I) Y_{0, s_{k} } Z_y^{0} \cr
& \quad + a_{\ell} Y_{s_{\ell},s_{\ell +1}}   E(I-1,I)  Y_{0, s_{\ell+1} } Z_y^{0},
\end{align*}
which finishes the proof.
\end{proof}
Since we are interested in $Z_y^{t}(2)$, we write a similar version of  Lemma \ref{vector}. In the following observations we consider $\ell$ such that
\[ s_1 <s_2< \ldots <s_{\ell}\leq t < s_{\ell+1}.\]
Using the fact that $Z_y^{t} = Y_{s_{\ell},t }  Z_y^{s_{\ell}},$ we get the following.
\begin{corollary}\label{coordinate}
We have that 
 \begin{align*}
Z_y^{t} &=Y_{0,t }  Z_y^{0}+ \sum_{k=1}^{\ell-1} a_k  Y_{s_k, t} E(I-1,I) Y_{0, s_{k}} Z_y^{0}  + a_{\ell} E(I-1,I)  Y_{0,t} Z_y^{0}
\end{align*}
and
$$Z_y^{t}(2) = [Y_{0, t }  Z_y^{0}](2)+ \sum_{k=1}^{\ell-1} a_k [ Y_{s_k, t} E(I-1,I) Y_{0, s_{k} } Z_y^{0}](2).$$

\end{corollary}
To study $[Y_{0, t }  Z_y^{0}](2)$, consider a vector process starting at $ y_{I} \cdot e_{I}$ and having the same updates as the original process. Then $[Y_{0, t }  Z_y^{0}](2)$ is the second coordinate of this process. 

Similarly, $[Y_{s_k, t } E(I-1,I) Y_{0, s_{k} } Z_y^{0}](2)$ is the same as the second coordinate of the vector process that starts at $y_{I}^{s_k} e_{I-1}$, where $y_{I}^{s_k}$ is the $I$--th coordinate of $Z_y^{s_k}$, and whose updates from are the same as the updates that occur between times $s_k$ and $t$.

Recall that our goal is to study the terms $\cos \left( \frac{2 \pi Z_y^{t}(2) }{m} \right)$ that appear in \eqref{y} in the form of $\lambda_{y,w_s}$. The following lemma introduces a condition under which $\vert Z_y^{t}(2)  \vert$ is guaranteed to be big during the interval $[t_{j+1}, t_{j+2}]$. This will be crucial to proving that the eigenvalues of the walk are sufficiently small for a constant fraction of the time. Recall that $D_{y}^I$ be the set of all $(y,I)$--good intervals by time $\overline{t}_{n,m}$. 

Let $t \in [t_{j+1}, t_{j+2}]$. We consider the decomposition of Corollary \ref{coordinate}
\[N^t= \begin{cases}
 [Y_{s_{j_1},t } E(I-1,I) Y_{0, s_{j_1} } Z_y^{0}](2), & \textup{ if } t_{j}\leq s_{j_1} \leq t_{j+1},\\
 0, & \textup{ otherwise,}
\end{cases}\]
where $s_{j_1}$ is the first time in $[t_j, t_{j+1}]$ that the $I$--th clock rings.
 
\begin{lemma}\label{de20}
Let $t\in [t_{j+1}, t_{j+2}]$. For every $x \in \{0, \ldots m/4\}$, we have that
\begin{equation}\label{DiaHough2}
\prcond{  Z_y^{t}(2)   \in [x,m/4]}{\mathcal{F}_{j}^*} \geq \frac{1}{4}\prcond{ 2 N^t \in [2x, m/2] }{\mathcal{F}^*_{j}} \}.
\end{equation}
\end{lemma}
\begin{proof}
Let $\mathcal{Y}_t$ be the event that $ 2 N^t \in [2x, m/2]$. 
\begin{align}
&\prcond{  Z_y^{t}(2)   \in [x,m/4]}{\mathcal{F}_{j}^*} \label{la2} \geq  \prcond{  Z_y^{t}(2)   \in [x,m/4] }{\mathcal{Y}_t, \mathcal{F}^*_j}  \prcond{ \mathcal{Y}_t }{\mathcal{F}^*_{j}}.
   \end{align}
We turn to the decomposition of Corollary \ref{coordinate}. The condition $\{2 N^t \in [2x, m/2]\}$ combined with the fact that $a_1=1$ or $ a_1=-1$ with probability $1/4$ results in $ Z_y^{t}(2)   \in [x,m/4]$ with probability $1/4$. Therefore,
\begin{align*}
\prcond{ Z_y^{t}(2)  \in [x,m/4]}{\mathcal{F}_{j}^*}& \geq \frac{1}{4}\prcond{ 2 N^t \in [2x, m/2] }{\mathcal{F}^*_{j}} .
\end{align*}
This finishes the proof.
\qedhere
\end{proof}
We move onto applying the results established above.

\section{The case where $m$ is a prime}\label{prime}
In this section $m$ is a prime number. Since the case $m=2$ was covered in \cite{PeresSly}, from now on in this section $m$ will denote an odd prime.

\subsection{The main lemma for $m$ prime}\label{strategy}
For $m$ prime, we can use the Diaconis--Hough lemma for the proof of Theorem \ref{mainc}. We state the lemma below.

\begin{lemma}(\cite[\textup{Theorem 3}]{Hough})\label{DH}

Let $Z_t$ be the configuration of the rightmost corner of the upper triangular random walk at time $t$. For any set $A \subset  \mathbb{Z}/ m \mathbb{Z}$ we have that
$$\Vert \pr{Z_t \in \cdot} - U\Vert_{T.V.} \leq \exp(-r t 2^{-n} m^{-\frac{2}{n-1}}),$$
where $U$ is the uniform measure on $G_n(m)$ and $r
$ is a universal constant.
\end{lemma}
Diaconis and Hough mainly treat the case where $n$ is fixed. Therefore, the term $2^n$ is not announced in their main result, but can be found in the proof of \cite[Proposition 22]{Hough}. 

We select an index $I \in \{1,\ldots, n-1\}$ so that the Lemma \ref{DH} gives useful estimates for $y \in W_I$ when applied on $G_{I-1}(m)$.
Our goal is to prove that $Z_y^{t}(2)$ is big often enough. For $y\in Q_I$, we apply Lemma \ref{DH} for $n=3$. Finally, we bound the eigenvalues in $P_2$ using the $n=3$ case.

 \subsection{The eigenvalues for $y \in W_I$}
Let $I=\min \{ \lfloor \sqrt{\log m}\rfloor ,n-1 \}$, so that $2^I \leq m^{\frac{2}{I}}$. Recall that $W_I= \left( \mathbb{Z}/ m \mathbb{Z}\right)^{n-1} \setminus \langle e_1, \ldots, e_{I-1} \rangle$ and let $y \in W_I$. The goal of this section is to prove that $\vert Z_y^{t}(2)  \vert \geq \frac{m}{8}$ for a constant fraction of $[0,\overline{t}_{n,m}]$.
We choose the length of each interval $[t_j, t_{j+1}]$ to be $L= m^{\frac{4}{I }}$.

Let $\mathcal{G}_{j,I}^y$ be the event $\{[t_j,  t_{j+1}] \in D^I_y  \}$ for the $I$ that was chosen, but for simplicity we will write $\mathcal{G}^y_{j}$. The following lemma is the main tool for proving the Theorem \ref{mainc}. 
\begin{lemma}\label{de}
For $y \in W_I$, we have 
\begin{equation}
\prcond{ \vert Z_y^{t}(2)  \vert \geq \frac{m}{8} }{\mathcal{F}^*_{j}} \geq \frac{1}{128} 1_{\mathcal{G}^y_{j}},
\end{equation}
for every $t \geq t_{j+1}$.
\end{lemma}
\begin{proof}
We apply Lemma \ref{de20}. 
In particular, we bound $\prcond{ \vert2 N^t\vert \geq \frac{m}{4}  }{\mathcal{F}^*_{j}}$ from below, using Lemma \ref{DH}, to get 
$$\prcond{ \vert 2 N^t\vert \geq \frac{m}{4}  }{\mathcal{F}^*_{j} }  \geq \left( \frac{1}{8}- \exp{ \bigg \lbrace -r \frac{L}{2^{I}m^{2/I}} \bigg \rbrace}\right) 1_{\mathcal{G}^y_{j}}\pr{t_{j} \leq s_{j_1} \leq t_{j+1}}.$$
By our choice of $L$ and $I$, we have
\begin{equation}\label{DiaHough}
\prcond{ \vert 2 N^t\vert \geq \frac{m}{4}  }{\mathcal{F}^*_{j}}  \geq  \frac{1}{16}1_{\mathcal{G}^y_{j}}\pr{t_j \leq s_{j_1} \leq t_{j+1}}
.
\end{equation}
We take $d_2$ big enough so that $\pr{t_j \leq s_{j_1} \leq t_{j+1}} \geq 1/2$.  Lemma \ref{de20} finishes the proof.
\end{proof}

The next lemma says that the event $\{\vert Z^s_y(2)  \vert >  m/8 \}$ considered in Lemma \ref{de} has a good chance of holding for a constant fraction of the time. Let $S$ be an appropriately chosen constant, which does not dependent on $j,m, n$ and $y$.
\begin{lemma}\label{B1}
Denote the event  
\[
B_{j,y} :=\Big\{\int_{t_{j+1}}^{t_{j+2}}1_{ \{\vert Z^s_y(2)  \vert >  m/8 \} } ds  \geq \frac{1}{16384} L\Big \}. 
\]
For $y \in W_I,$ we have 
$$\prcond{B_{j,y}} {\mathcal{F}^*_{j}} \geq \frac{1}{129}1_{\mathcal{G}^y_{j}},
$$
for every $j \geq 1$.
\end{lemma}

\begin{proof}
Let $R_j$ count the length of $s \in [t_{j+1},  t_{j+2}]$, such that $\vert Z^{s}_y(2) \vert \leq   m/8$. So $R_j= \int_{t_{j+1}}^{ t_{j+2}} 1_{ \lbrace \vert Z^{s}_y(2) \vert \leq  m/8  \rbrace}ds $. Lemma \ref{de} and Fubini give
\begin{align*}
& \expect{ R_j  \vert \mathcal{F}^*_{j}}\leq  \left( 1-\frac{1}{128}1_{\mathcal{G}^y_{j}} \right) L.
\end{align*}
 Markov's inequality then gives 
\begin{align*}
& 1_{\mathcal{G}^y_{j}} \prcond{ R_j  > \frac{16383}{16384} L  }{\mathcal{F}^*_{j}} \leq \frac{128}{129} .\end{align*}
Therefore,
\begin{align*}
&\prcond{B_{j,y}} {\mathcal{F}^*_{j}} \geq \prcond{ R_j  \leq \frac{129}{128}\left( 1-\frac{1}{128}1_{\mathcal{G}^y_{j}} \right) L  }{\mathcal{F}^*_{j}} \geq \frac{1}{129}1_{\mathcal{G}^y_{j}} .
\end{align*}
\end{proof}

For the next lemma, we set $A=16384$.

\begin{lemma}\label{r1}
 Recall that $P^t$ is the indicator function that the clock of the second row rings at time $t$ and $A^t_{y,x}= \int_0^t  1_{\lbrace \vert Z_y^s(2)   \vert >   x\rbrace } dP^s$. Let $x= m /8$. For $y \in W_I$, letting $B_{y}= \lbrace A^{\overline{t}_{n,m}}_{y,x}
 \geq (400A)^{-1} \overline{t}_{n,m} \rbrace $, we have
 \begin{align*}
 \pr{B_{y}} \geq   1- \frac{b}{m^{gn}} e^{-c},
 \end{align*}
 where $b,g$ are suitable constants with $g>1$.
 \end{lemma}
 \begin{proof}
 Recalling Lemma \ref{r115} and applying Lemma \ref{B1} we can take $ \zeta=1/129 $. We have that there is a constant $g>1$ such that $\frac{\zeta^2 \overline{t}_{n,m}}{400^2 L_1} \geq g n \log m +c$. Lemmas \ref{r15} and \ref{B12} give the desired result.
 \end{proof}

The rest of the eigenvalues will be studied in Section \ref{general}.

\section{The case where $m$ is not a prime}\label{general}
In this section, we study the quantity $Z_y^t(2)$ for the case where $m$ is not necessarily prime. The strategy is similar to the $m$ prime case, however we can no longer apply Lemma \ref{DH} directly. This is why we start by proving a lemma similar to Lemma \ref{DH}, which works for $m \in \mathbb{N}$ that is not necessarily prime. 

Let $J= \min \{ \lfloor (\log m)^{1/3} \rfloor, n-2 \}$ and $h=20(J+1)$. Let $A_{J,h,m}= m^{2/J}/ (6^{2/J}2\log h)$ and let $p$ be a prime such that $\frac{1}{20} A_{J,h,m} \leq 6 r^{-1}2^J p^{2/J} \leq \frac{1}{2} A_{J,h,m}$. In other words, we choose a prime $p$ such that $$\left( \frac{r}{120 \log h}\right)^{J/2} 2^{-(J+J^2)/2} \frac{m}{6} \leq p \leq \left( \frac{r}{6 \log h}\right)^{J/2} 2^{-(J+J^2)/2} \frac{m}{6}.$$ Recall that $r$ is the constant from Lemma \ref{DH}. The goal of this section is to prove the following lemma, which ensures that $\vert aZ_t(n-J-1)\vert  $ has a good chance of being big. This will be crucial to proving that $\vert Z_y^t(2)\vert $ has a good chance of being big as well, forcing 
the eigenvalue $\cos \left( \frac{2 \pi Z_y^t(2) }{m} \right)$ to be bounded away from one.

\begin{lemma}\label{BH2}
Let $Z_t$ be the the last column of $X_t$. There is an absolute constant $K$ such that 
    \[\pr{\vert aZ_t(n-J-1)\vert > m e^{-K (\log m)^{2/3}} } \geq e^{- ( \log m)^{1/3}},\]
for any $t \in [6r^{-1}2^Jp^{2/J}, A_{J,h,m} ]$ and $a \in \{1,\ldots, m-1\}$.
\end{lemma}

To prove Lemma \ref{BH2}, we will need the following lemma concerning $Z_t$ over~$\mathbb{Z}$.
\begin{lemma}\label{AH}
Let $W_t= X_t e_{n-1}$ be the column process over $\mathbb{Z}$ which starts at $(0,\ldots,1)^T$. Let $\overline t \in [0, t_{n,m}]$. Then, we have that 
        $$\pr{\max_{\substack{t \leq \overline t,\\ 0\leq i \leq k}} \{ \vert W_t(n-i)\vert  \} \leq \overline t^{k/2} (2\log h)^{k/2}} \geq \left(1- \frac{3}{h}  \right)^k,$$
        for $k \leq J+1 \leq n-1$.
\end{lemma}
\begin{proof}[Proof of Lemma \ref{AH}]
We prove the result by induction on $k$. 
Denote the event
$$\mathcal{A}_k=\bigg \{\max_{\substack{t \leq \overline t,\\ i \leq k}} \{ \vert W_t(n-i)\vert  \} \leq \overline t^{k/2} (2\log h)^{k/2} \bigg \}.$$
For $k=0$, clearly $\pr{\mathcal{A}_0}=1$.
We take as our assumption hypothesis the assumption $\pr{\mathcal{A}_k} \geq \left(1- \frac{3}{h}  \right)^k$. 

Let $R(t)$ be the number of times the $n-k$ clock rings by time $t$ and let $s_1\leq s_2 \leq \ldots \leq s_{R(t)}\leq t$ be the times that the $n-k$ clock rings. The tails of a Poisson distribution with mean $t$ give us that $\pr{R(t)\geq \frac{2t}{3}} \leq (1.2)^{-t/3}$.

Set $M_t:=W_y^{t}(n-k-1)= \sum_{i=1}^{R(t)} a_i W_y^{s_i}(n-k)$, where the $\{a_i \in \{0, \pm 1\}\}$ are the random signs, and note that it is a martingale. Let $\tau= \inf \lbrace t: |M_t| \geq \overline t^{(k+1)/2} (2\log h)^{(k+1)/2} \rbrace $.
Then $M_{t \wedge \tau \wedge \frac{2t}{3}}$ is also a martingale.
The Azuma-Hoeffding inequality gives that
\begin{align*}
    \pr{\vert M_{\overline t \wedge \tau \wedge \frac{2t}{3}} \vert \geq  \overline t^{(k+1)/2} (2\log h)^{(k+1)/2} \vert \mathcal{A}_k} \leq 2 e^{- \frac{\overline t^{k+1} (2\log h)^{k+1}}{2 \overline t^{k+1} (2\log h)^{k}}}= \frac{2}{h}. 
\end{align*}
This gives
$$\prcond{\mathcal{A}^c_{k+1}}{\mathcal{A}_{k}}\leq \pr{R(\overline t)\geq \frac{2t}{3}} + \pr{\vert M_{\overline t \wedge \tau \wedge \frac{2t}{3}} \vert \geq  \overline t^{(k+1)/2} (2\log h)^{(k+1)/2} \vert \mathcal{A}_k} \leq \frac{3}{h} , $$
since the choice of $\overline{t}, h$ and $J$ give that $(1.2)^{-\overline t/3} \leq \frac{1}{h}$. Therefore,
$$\pr{\mathcal{A}_{k+1}} \geq \left(1-\frac{3}{h}\right)\pr{\mathcal{A}_{k}} ,$$
which gives the desired result.

\end{proof}
Let $T$ be the first time that there is a $j \leq n- J-1$ which satisfies $\vert W_t(j)\vert > m/6$. Setting $\overline t=A_{J,h,m} :=m^{2/J}/ (6^{2/J}2\log h)$ and $k=J+1$, Lemma \ref{AH} says that 
\begin{equation}\label{time}
\pr{T >A_{J,h,m}} \geq \left( 1- \frac{3}{20(J+1)}\right)^{J+1} \geq e^{-\frac{1}{5}}\geq \frac{4}{5}.    
\end{equation}

Let 
\[\theta_k(t):= \max_{\substack{A \subset \mathbb{Z}/ m \mathbb{Z} \\ \vert A \vert \leq k}} \bigg \{\pr{Z_t(n-I) \in A} \bigg \}.\]
The following lemma uses Lemma \ref{DH} to give bounds for $\tilde Z_t$ over $\mathbb{Z}/p \mathbb{Z}$ and argues (among other things) that $Z_t(n-I)$ has a good chance of being bounded away from zero for appropriate values of $t$. Recall that $p$ is a prime such that $\frac{1}{20} A_{J,h,m} \leq 6 r^{-1}2^J p^{2/J} \leq \frac{1}{2} A_{J,h,m}$. 

\begin{lemma}\label{important1}
For $t \in [6r^{-1}2^J p^{2/J}, A_{J,h,m} ]$, we have 
$$\theta_{p/3}(t) \leq 3/5.$$
\end{lemma}
\begin{proof}
Let $\Tilde{Z}_t$ be the process over $\mathbb{Z}/p \mathbb{Z}$ and let 
$$\Tilde{\theta}_k(t):= \max_{\substack{A \subset \mathbb{Z}/ p \mathbb{Z} \\ \vert A \vert \leq k}} \bigg \{\pr{\Tilde{Z}_t(n-J-1) \in A} \bigg \}.$$
Recall that $W_t$ is the column process over $\mathbb{Z}$, just as in Lemma \ref{AH}.
If $T> t$ then $W_t(n-J-1)$ has not left $[-m/6, m/6]$. This allows us to couple $Z_t(n-J-1)$ with $W_t(n-J-1)$.  Let $A \subset \mathbb{Z}/ m \mathbb{Z}$ be a set of magnitude at most $p/3$ where $\theta_{p/3}(t)$ is realized at. The coupling of $Z_t(n-J-1)$ with $W_t(n-J-1)$ allows us to view $A$ as a subset of $\mathbb{Z}$. Projecting $A$ to $ \mathbb{Z}/ p \mathbb{Z}$ gives that $\pr{Z_t \in A, T>t} \leq \pr{\Tilde{Z}_t \in B}$, where $B \subset \mathbb{Z}/ p \mathbb{Z}$ is the projection of $A$. To sum up,
$$\theta_{p/3}(t) \leq \Tilde{\theta}_{p/3}(t) + \pr{T \leq t}.$$
Equation \eqref{time} gives
$$\theta_{p/3}(t) \leq \Tilde{\theta}_{p/3}(t) + \frac{1}{5}.$$
We will now prove that $\Tilde{\theta}_{p/3}(t) < 2/5 $. If we instead assume that $\Tilde{\theta}_{p/3}(t) \geq 2/5 $ then there is a set $A \subset \mathbb{Z}/ p \mathbb{Z}$ with $\vert A \vert \leq p/3$ such that $$\pr{\Tilde{Z}_t(n-J-1) \in A}- \pi_p(A) \geq 1/15,$$
where $\pi_p$ is the uniform measure over $\mathbb{Z}/ p \mathbb{Z}$. This implies that 
\begin{equation}\label{c1}
d_{T.V.}(\Tilde{Z}_t(n-J-1), \pi_p) \geq 1/15 .
\end{equation}
However, $t \geq 6r^{-1}2^J p^{2/J},$ and therefore Lemma \ref{DH} gives that
\begin{equation}\label{c2}
    d_{T.V.}(\Tilde{Z}_t(n-J-1), \pi_p) \leq e^{-3} .
\end{equation}
Equation \eqref{c2} contradicts \eqref{c1}.
Therefore, $\Tilde{\theta}_{p/3}(t) < 2/5 ,$ and $$\theta_{p/3}(t) \leq \Tilde{\theta}_{p/3}(t) + 1/5\leq 3/5,$$
which finishes the proof.
\end{proof}
The following corollary uses the definition of $p$ to quantify how big $Z_t(n-J-1)$ can be.

\begin{corollary}\label{important3}
For $t \in [6r^{-1}2^J p^{2/J}
, A_{J,h,m} ]$, we have that there is a universal constant $K$ such that 
\[\pr{\vert Z_t(n-J-1)\vert > m e^{-K (\log m)^{2/3}} } \geq 2/5.\]
\end{corollary}
\begin{proof}
Lemma \ref{important1} gives that 
\[\pr{\vert Z_t(n-J-1)\vert > p/6} \geq 2/5.\]
The fact that $\frac{1}{20} A_{J,h,m} \leq 6
r^{-1}2^J p^{2/J} $ gives that $$\frac{p}{6} \geq m e^{-K (\log m)^{2/3}},$$
where $K$ is a universal constant. This completes the proof.

\end{proof}
We are now ready to prove Lemma \ref{BH2}.
\begin{proof}[Proof of Lemma \ref{BH2}]
Let $a \in  \mathbb{Z}$ and let 
$$\theta^a_k(t):= \max_{\substack{A \subset \mathbb{Z}/ m \mathbb{Z} \\ \vert A \vert \leq k}} \{\pr{aZ_t(n-J-1) \in A}\}.$$
Let $g=\gcd (a,m) $. If $g=1$ then $\theta^a_k(t)=\theta_k(t)$ and the statement follows from Corollary \ref{important3}. 

If $g \neq 1$ then let $m'= m/g$, $a'=a/g$, $J'= \min \{(\log m')^{1/3}), n-1\}$ and $h'=20(J'+1)$. Then we can view $\overline{Z}_t:= aZ_t(n-J-1)$ as a process over $\mathbb{Z}/m' \mathbb{Z}$.

Let $s$ be a suitably chosen universal constant.
If $m' \geq s,$ then let $p'$ be such that $\frac{1}{20} A_{J',h',m'} \leq 6r^{-1}2^J (p')^{2/J} \leq \frac{1}{2} A_{J',h',m'}$. We therefore have that there exists a constant $\tilde r$, that does not depend on $m'$, such that $p' \geq 6^{-1} \tilde{r}^{J'/2}  2^{-J'^2/2}(h')^{-J/2} (\log h')^{-J'/2}m' $. Corollary \ref{important3} says that 
\[\pr{\vert a \overline{Z}_t)\vert > p'/6} \geq 2/5,\]
for $t \in [6\tilde r^{-1}2^{J'} p^{\prime2/J'}
, A_{J',h',m'} ]$. Therefore 
\[\pr{\vert aZ_t(n-J'-1)\vert > gp'/6} \geq 2/5.\]
Using the fact that $gp' \geq 6^{-1} \tilde{r}^{J'/2}  2^{-(J')^2/2}h'^{-J'/2} (\log h')^{-J'/2}m $ we get 
\[\pr{\vert aZ_t(n-J'-1)\vert >6^{-1} \tilde{r}^{J'/2}  2^{-(J')^2/2}h'^{-J'/2} (\log h')^{-J'/2}m } \geq 2/5,\]
for $t \in [6\tilde r^{-1}2^{J'} p^{\prime2/J'}
, A_{J',h',m'} ]$.
The choice of $J'$ and $h'$ give 
\[e^{-K (\log m)^{2/3} } \leq 6^{-1} \tilde{r}^{J'/2}  2^{-(J')^2/2}h'^{-J'/2} (\log h')^{-J'/2},\]
for a suitable constant $K$. This implies that 
\[\pr{\vert aZ_t(n-J'-1)\vert >m e^{-K (\log m)^{2/3} }} \geq 2/5,\]
for $t \in [6\tilde r^{-1}2^{J'} p^{\prime2/J'}
, A_{J',h',m'} ]$.

We are going to investigate what happens for all times $t \in [A_{J',h',m'} , A_{J,h,m} ]$.
Consider the event $\mathcal{C}$ that in the time interval $[A_{J',h',m'} , A_{J,h,m} ]$ there are times $c_{n-J'-1}<  \ldots < c_{n-J-2} $ during which clocks $n-J'-1,\ldots,n-J-2$ ring respectively. The expected time it takes to see such a sequence of updates is at most $J'-J$. Markov's inequality gives
$\pr{C} \geq 1- e^{-(\log m)^{2/3}}$.

Let $C_{n-J'-1}$ be the first time after $A_{J',h',m'}$ that the $n-J'-1$ clock rings. Let $i \in \{n-J', \ldots , n-J-2\}$ and let $C_i= \inf \{ t > C_{i-1}: \textup { clock $n-i$ rings}\}$. We now consider the event $D_i=\{ \textup{clock }i \textup{ does not ring in }[C_{i}, C_{i+1}]\}.$ Let $\mathcal{D}=\cap D_i$.


Then $\pr{\mathcal{D} \vert \mathcal{C}} \geq 2^{J-J'}$. If $\vert aZ_t(n-J'-1)\vert >me^{-K (\log m)^{2/3} }$ then at least one among adding or subtracting $\vert aZ_t(n-J'-1)\vert$ to $\vert aZ_t(n-J-1)\vert$ will result in $\vert aZ_t(n-J-1)\vert >m/6$. Therefore, for $t \in [A_{J',h',m'} , A_{J,h,m} ]$
\[\pr{\vert aZ_t(n-J-1)\vert >m/6 } \geq 2^{J-J'-1}(1- e^{-(\log m)^{2/3}})\geq 2^{-(\log m)^{1/3}}(1- e^{-(\log m)^{2/3}}).\]

If $m'\leq s$, then the walk $aZ_t(n-J-1)$ on $\mathbb{Z}/m' \mathbb{Z}$ mixes in a bounded number of steps. Similarly to before, we have 
\[\pr{\vert aZ_t(n-J-1)\vert > 6^{-1} \tilde{r}^{J/2}  2^{-J^2/2}h^{-J/2} (\log h)^{-J/2}m } \geq 2/5.\]
Using the fact that there is a universal constant $K$ such that 
\[6^{-1} \tilde{r}^{J/2}  2^{-J^2/2}h^{-J/2} (\log h)^{-J/2}m  \geq m  e^{-K (\log m)^{2/3}},\]
we conclude the proof.
\end{proof}

\subsection{The eigenvalues $y\in W_I$}
Recall that $W_I=\left( \mathbb{Z}/ m \mathbb{Z}\right)^{n-1} \setminus \langle e_1, \ldots, e_{I-1} \rangle$. We are now going to consider the decomposition proved in Corollary \ref{coordinate}. For the definition of the $(y,I)$--good intervals, we are going to consider $I=\min \{ J+3, n-1 \}$, where $J= \lfloor (\log m)^{1/3} \rfloor$. Recall that $A_{J,h,m}= (m/6)^{2/J}/ (2\log h)$. Let $L_1= A_{J,h,m}$ be the length of each $(y,I)$--good interval. Recall that $\mathcal{G}_{j}^y$ is the event $\{[t_j,  t_{j+1}] \in D^I_y  \}$, i.e. that the interval $[t_j,  t_{j+1}]$ is $(y,I)$--good.

The following lemma is one of the main tools for proving Proposition \ref{q}. 
\begin{lemma}\label{deg}
For $y \in  W_I$, we have 
\begin{equation}
\prcond{ \vert Z_y^{t}(2)  \vert \geq m  e^{-K (\log m)^{2/3}} }{\mathcal{F}_j^*}  \geq \frac{1}{4}e^{- ( \log m)^{1/3}}1_{\mathcal{G}^y_{j}},
\end{equation}
for every $t \geq t_{j+1}$.
\end{lemma}
\begin{proof}
Using the decomposition
$$Z_y^{t}(2) =[ Y_{0, t }  Z_y^{0}](2)+ \sum_{k=1}^{\ell-1} a_k  [Y_{s_k, t} E(I-1,I) Y_{0, s_{k} } Z_y^{0}](2),$$
as presented in Corollary \ref{coordinate}, we notice that $N_t=   [Y_{s_{j_1}, t} E(I-1,I) Y_{0, s_{j_1} } Z_y^{0}](2)$ has the form $a Z_{t}(N-J-1)$ for $N=J+2$ and $a=[Y_{0, s_{j_1} } Z_y^{0}](I)$.

Set
$$Z =[ Y_{0, t }  Z_y^{0}](2)+ \sum_{\substack{k=1\\ k \neq s_{j_1}}}^{\ell-1} a_k  [Y_{s_k, t} E(I-1,I) Y_{0, s_{k} } Z_y^{0}](2).$$

If $\vert N_t - \frac{m}{2} \vert \leq m  e^{-K (\log m)^{2/3}}  $ and $\vert Z\vert \geq  m  e^{-K (\log m)^{2/3}} $ then the event $\{a_{s_{j_1}}=0\}$  implies that $\vert Z_y^{t}(2)  \vert \geq m  e^{-K (\log m)^{2/3}} $. This happens with probability at least $ \frac{1}{2} 1_{\mathcal{G}^y_{j}} \pr{t_j \leq s_{j_1} \leq t_{t_{j+1}}}$.

If $\vert N_t - \frac{m}{2} \vert \leq m  e^{-K (\log m)^{2/3}}  $ and $\vert Z\vert \leq  m  e^{-K (\log m)^{2/3}} $ then the event $\{a_{s_{j_1}}=\pm 1\}$  implies that $\vert Z_y^{t}(2)  \vert \geq m  e^{-K (\log m)^{2/3}} $. This happens with probability at least $ \frac{1}{2} 1_{\mathcal{G}^y_{j}} \pr{t_j \leq s_{j_1} \leq t_{t_{j+1}}}$.

Otherwise, by our choice of $L$ and $J$, Lemma \ref{BH2} gives that 
$$\prcond{ \vert 2 N^t\vert \geq m  e^{-K (\log m)^{2/3}}  }{\mathcal{F}^*_{j} }  \geq  e^{- ( \log m)^{1/3}} 1_{\mathcal{G}^y_{j}} \pr{t_j \leq s_{j_1} \leq t_{t_{j+1}}}.$$
In this case, the claim follows from Lemma \ref{de20}.

\end{proof}
The following lemma proves Lemma \ref{r115} for $y \in W_I$. 
\begin{lemma}\label{B12}
Let $B_{j,y}$ denote the event 
\[
B_{j,y} :=\Big\{\int_{t_{j+1}}^{t_{j+2}}1_{ \{\vert Z^s_y(2)  \vert > m  e^{-K (\log m)^{2/3}}\} } ds  \geq \frac{1}{16} e^{-2 (\log m)^{1/3}} L\Big \}. 
\] 
For $y \in W_I ,$ we have that
$$\prcond{B_{j,y}} {\mathcal{F}^*_{j}} \geq \frac{1}{5}e^{- ( \log m)^{1/3}}1_{\mathcal{G}^y_{j}},$$
for every $j \geq 1$.
\end{lemma}

\begin{proof}
Let $R_j$ count the length of $s \in [t_{j+1},  t_{j+2}]$, such that $\vert Z^{s}_y(2) \vert \leq  m  e^{-K (\log m)^{2/3}}.$ So $R_j= \int_{t_{j+1}}^{ t_{j+2}} 1_{ \lbrace \vert Z^{s}_y(2) \vert \leq   m  e^{-K (\log m)^{2/3}} \rbrace}ds $. Lemma \ref{deg} gives that
\begin{align*}
&\expect{ R_j  \vert \mathcal{F}^*_{j}}\leq  \left( 1-\frac{1}{4}e^{- ( \log m)^{1/3}}1_{\mathcal{G}^y_{j}} \right)L_1.
\end{align*}
 Markov's inequality then gives 
\begin{align*}
&1_{\mathcal{G}^y_{j}} \prcond{ R_j  > \left( 1- \frac{1}{16} e^{-2 (\log m)^{1/3}} \right) L_1  }{\mathcal{F}^*_{j}} \leq \frac{1}{1+\frac{1}{4}e^{- ( \log m)^{1/3}}}.\end{align*}
The constant $S$ is chosen so that
\begin{align*}
\prcond{B_{j,y}} {\mathcal{F}^*_{j}} & \geq \prcond{ R_j  \leq \frac{1}{16} e^{-2 (\log m)^{1/3}} L_1  }{\mathcal{F}^*_{j}}\\ & \geq \frac{\frac{1}{4}e^{- ( \log m)^{1/3}}}{1+\frac{1}{4}e^{- ( \log m)^{1/3}}}1_{\mathcal{G}^y_{j}}\\ & \geq \frac{1}{5}e^{- ( \log m)^{1/3}}1_{\mathcal{G}^y_{j}}.\qedhere
\end{align*}
\end{proof}

For the next lemma, we set $A=16 e^{2 (\log m)^{1/3}} $. The following lemma makes use of the bound that Lemma \ref{r15} provides for $y \in W_I$.  Recall that $P^t$ is the point process of clock rings of the second and $A^t_{y,x}= \int_0^t  1_{\lbrace \vert Z_y^s(2)   \vert >   x\rbrace } dP^s$. 
\begin{lemma}\label{rr1}
 Let $x= m  e^{-K (\log m)^{2/3}}$. For $y \in W_I$, letting $$B_{y}= \lbrace A^{\overline{t}_{n,m}}_{y,x}
 \geq (400A)^{-1} \overline{t}_{n,m} \rbrace,$$ we have
 \begin{align*}
 \pr{B_{y}} \geq   1- \frac{b}{m^{Rn}} e^{-c},
 \end{align*}
 where $b$ is a suitable constant and $R>1$.
 \end{lemma}
 \begin{proof}
 Recalling Lemma \ref{r115} and applying Lemma \ref{B12} we can take $ \zeta=\frac{1}{5}e^{- ( \log m)^{1/3}} $. We have that there is a constant $R>1$ such that $\frac{\zeta^2 \overline{t}_{n,m}}{400^2 L_1} \geq R  n  \log m +c$. Lemmas \ref{r15} and \ref{B12} give the desired result.
 \end{proof}

\subsection{The eigenvalues $y\in  Q_I $}

Recall that $Q_I= \langle e_1, \ldots, e_{I-1} \rangle \setminus \langle e_1 , e_2\rangle $. Here we will use the fact that $J=3$ and $ p \geq m/ \tilde K$, where $\tilde K$ is a constant universal on $n$ and $m$. The proof of the following result is identical to the proof of Corollary \ref{important3}.
\begin{corollary}\label{important4}
For $t \in [48r^{-1} p^{2/3}
, A_{3,h,m} ]$, we have that there is a universal constant $\tilde{K}$ such that 
\[\pr{\vert Z_t(n-4)\vert > m / \tilde{K}} \geq 1/2.\]
\end{corollary}
We now want to bound $\vert Z_y^{t}(2)  \vert$. Let $L_2
= m$ be the length of each $(y,I)$--good interval.
\begin{lemma}\label{deg2}
For $y \in  Q_I$, we have that
\begin{equation}
\prcond{ \vert Z_y^{t}(2)  \vert \geq m /\tilde{K} }{\mathcal{F}_{j}^*} \geq \frac{1}{8}1_{\mathcal{G}^y_{j}},
\end{equation}

for every $t \geq t_{j+1}$.
\end{lemma}
The proof of Lemma \ref{deg2} is similar to the proof of Lemma \ref{deg}. Similarly to Lemma \ref{B12}, we have the following lemma.
\begin{lemma}\label{B13}
Let $B_{j,y}$ denote the event that $\vert Z^s_y(2)  \vert > m  / \tilde{K}$  for at least $1/64 $ of 
$ [ t_{j+1}, t_{j+2}]$. For $y \in Q_I ,$ we have 
$$\prcond{B_{j,y}} {\mathcal{F}_{j}^*} \geq \frac{1}{9} 1_{\mathcal{G}^y_{j}},$$
for every $j \geq 1$.
\end{lemma}
For the next lemma, we set $A= 64$.
\begin{lemma}\label{r2}
 Recall that $P^t$ is the indicator function that the clock of the second row rings at time $t$ and $A^t_{y,x}= \int_0^t  1_{\lbrace \vert Z_y^s(2)   \vert >   x\rbrace } dP^s$. Let $x=m / \tilde{K}.$ For $y \in Q_I$,
letting $\tilde{D}_{y}= \lbrace A^{\overline{t}_{n,m}}_{y,x}
\geq (400A)^{-1} \overline{t}_{n,m} \rbrace $, we have
\begin{align*}
\pr{\tilde{D}_{y}} \geq 1- b \frac{1}{n^{gm}} e^{-c},
\end{align*}
where $b,g$ are suitable constants with $g>1$.
\end{lemma}
\begin{proof}
Recalling Lemma \ref{r115} and applying Lemma \ref{B13} we can take $ \zeta=1/65 $. We have that there is a constant $R$ such that $\frac{\zeta^2 \overline{t}_{n,m}}{L_2} \geq R m \log n +c$. Lemmas \ref{r15} and \ref{B13} give the desired result.
\end{proof}

\subsection{The eigenvalues $y\in  P_2 $}
Recall that $P_2= \langle e_1,e_2 \rangle \setminus \langle e_1 \rangle $. For $y \in P_2$, we write $y= ae_1 +b e_2$ with $b \neq 0$. Therefore, we observe that $Z_y^t(2)= a+ b S^t$, where $S^t$ is a lazy, simple random walk on the cycle $\mathbb{Z}/m \mathbb{Z}$ starting at zero. In this section, we consider the length of the intervals to be $L_3=  \lfloor \log m \rfloor$.

\begin{lemma}
 Let $\mathcal{I }\subset \mathbb{Z} / m \mathbb{Z}$ with $\vert \mathcal{I} \vert= \lfloor \sqrt{\log m} \rfloor$. For every $y \in P_2$ and $t\geq t_{j+1}$, we have 
 \[\pr{ Z_y^t(2)-Z_y^{t_j}(2) \notin \mathcal{I} } \geq 1/2 .\]
\end{lemma}
\begin{proof}
Writing $y= ae_1 +b e_2$, we have $Z_y^t(2)-Z_y^{t_j}(2)= b( S^t- S^{t_j})$.
Assume for a contradiction that 
\begin{equation}\label{opp}
 \pr{ b(S^t- S^{t_j}) \notin \mathcal{I} } < 1/2 .
 \end{equation}
Let $Q^t$ be the transition matrix of $b S^t$. We have 
\begin{align}
\label{qel2}\Vert Q^{t-t_j}- \pi \Vert_2^2 &=  \sum_z  \frac{1}{m}\vert m  Q^{t-t_j}_0(z)-1 \vert^2  \cr
&= \sum_z m ( Q^{t-t_j} (z))^2-1  \cr
& \geq \sum_{z \in \mathcal{I}  } m ( Q^{t-t_j} (z))^2-1  .
\end{align}
Cauchy--Schwartz leads to
\begin{align}
\eqref{qel2} &  \geq \frac{ m}{\vert \mathcal{I} \vert}  \left( \sum_{ z \in  \mathcal{I} } Q^{t-t_j} (z)\right)^2-1 \cr
& \label{f12} \geq \frac{ m}{4\vert \mathcal{I} \vert}-1\\
& \label{dr} = \frac{ m}{4\sqrt{\log m}}-1,
\end{align}
where \eqref{f12} occurs by applying \eqref{opp}. Let $g$ be the $\gcd$ of $b$ and $m$. Given that $bS^t$ can be viewed as simple random walk on $\mathbb{Z}/ g \mathbb{Z}$, we have 
\begin{align}
\Vert Q^{t-t_j}- \pi \Vert_2^2 & \leq  \sum_{y=1}^{m/g -1}  e^{ -2 \left( t-  \sum_{i=1}^t \cos \frac{2 \pi g i }{m} \right) }\cr
& \label{expl} \leq \frac{m}{g}   e^{-2t} + \frac{\sqrt{3} m}{ 2g\sqrt{2 \pi  (t-t_j)}}\\
&\label{he} \leq \frac{m}{g} e^{-2L_3} + \frac{\sqrt{3} m}{ 2g\sqrt{2 \pi  L_3}},
\end{align}
where \eqref{expl} is a straightforward application of Lemma \ref{integral}.
Equation \eqref{he} contradicts \eqref{dr} for a suitable choice of the constant $\delta_1$ and this completes the proof.
\end{proof}
 This implies the following corollary.
 \begin{corollary}
 For every $y \in P_2$ and $t\geq t_{j+1}$, we have that
 \[\pr{\vert Z_y^t(2)\vert \geq \sqrt{ \log m}/2} \geq 1/4 .\]
 \end{corollary}
We follow the reasoning of the previous section to conclude the following lemmas.
\begin{lemma}\label{B23}
For $y \in P_2$, let $B_{j,y}$ denote the event that $\vert Z^s_y(2)  \vert > \sqrt{ \log m}/2$  for at least one third of 
$ [ t_{j+1}, t_{j+2}]$. For $y \in P_2 ,$ we have that
$$\prcond{B_{j,y}} {\mathcal{F}^*_{j}} \geq \frac{1}{2},$$
for every $j \geq 1$.
\end{lemma}
For the next lemma, we set $A=3$.
\begin{lemma}\label{r4}
 Recall that $P^t$ is the indicator function that the clock of the second row rings at time $t$ and $A^t_{y,x}= \int_0^t   1_{\lbrace \vert Z_y^s(2)   \vert >   x\rbrace } dP^s$ and let $x=\sqrt{\log m}/2$. For $y \in P_2$,
letting $\tilde{D}_{y}= \lbrace A^{\overline{t}_{n,m}}_{y,x}
\geq (400A)^{-1} \overline{t}_{n,m} \rbrace ,$ we have 
\begin{align*}
\pr{\tilde{D}_{y}} \geq 1-    \frac{b}{n^{gm}} e^{-c},
\end{align*}
where $b$ and $g>1$ are suitable constants.
\end{lemma}
\begin{proof}
There is a constant $g>1$ such that $\frac{\overline{t}_{n,m}}{L_3} \geq g m \log n +c$. Lemmas \ref{r15} and \ref{B12} give the desired result.
\end{proof}



 

\section{The proof of Proposition \ref{q}}
\begin{proof}[Proof of Proposition \ref{q}]
We first consider the case where $m$ is not prime. Let $x=m  e^{-K (\log m)^{2/3}},$ $w=m / \tilde{K} $ and $I=\min \{ J+3, n \}$, where $J= \lfloor (\log m)^{1/3} \rfloor$. For the first statement of Proposition \ref{q}, we need to specify these indices, but it is only for the second part
that we will justify these values, using Lemmas \ref{r1}, \ref{r2}, \ref{r4}. We also consider 
$$\overline t_{n,m}=  \gamma (m^2 n \log n + n^2 e^{\delta (\log m)^{2/3}}) + cnm^2 \log \log n,$$
which satisfies 
\begin{equation}\label{fb}
\overline t_{n,m} \geq  \gamma  n^2  L_1 e^{K (\log m)^{2/3}} \log m + cnm^2 \log \log n
\end{equation}
and  
\begin{equation}\label{sb}
\overline t_{n,m}  \geq \gamma n (\log m)^{4/3}+ cnm^2 \log \log n.
\end{equation}

Recall the definition of the event $E_{\overline t_{n,m},x_y}$ from Definition \ref{event}. Using the classical $\ell^2$ bound as described in \eqref{y}, given that $k,  w_1, \ldots w_k$ are such so that $E_{\overline t_{n,m},x_y}$ is satisfied, we have that for every $y \in W_I$ it is the case that $\vert Z_y^t(2) \vert \geq  m  e^{-K (\log m)^{2/3}}$ for at least $\frac{1}{6400} e^{-2 (\log m)^{1/3}} \overline t_{n,m}$. Then, \eqref{y} says that
\begin{align*}
 4 \Vert q_{\substack{k, \uw}}-u \Vert_{T.V.}^2  \cdot 1_{E_{\overline t_{n,m},I,x_y}}   & \leq 
\sum_{y \neq 0} \exp \left( -2 \left( k-  \sum_{i=1}^k \cos \frac{2 \pi (y^T w_i(2))}{m} \right) \right) . 
\end{align*}
The definition of $E_{\overline t_{n,m},x_y}$ gives that
\begin{align}
 4 \Vert q_{\substack{k, \uw}}-u \Vert_{T.V.}^2  \cdot 1_{E_{\overline t_{n,m},I,x_y}}   \leq  & \sum_{\substack{y \neq 0 \cr y \in \langle e_1 \rangle}}  e^{ -2 \left( k-  \sum_{i=1}^k \cos \frac{2 \pi y }{m} \right) } +   \sum_{ y \in P_2}  e^{ -2  v \overline t_{n,m} \left(1-  \cos \frac{ \tilde{\beta}\sqrt{ \log m} }{ m} \right) }  \cr
 & + \quad \sum_{y \in W_I} e^{ -2 e^{-v (\log m)^{1/3}}\overline t_{n,m}\left(1-  \cos  \left(2 \pi e^{-K (\log m)^{2/3}} \right)\right)} \cr
 & + \sum_{y \in Q_I} e^{ -v \overline t_{n,m}\left(1-  \cos  \left(2 \pi/ \tilde{K} \right)\right)},
    \label{30}
\end{align}
where $v$ is a universal constant on $n,m$. 
Equation \eqref{spl} gives that
\begin{align*}
 \sum_{\substack{y \neq 0 \cr y \in \langle e_1 \rangle}}  e^{ -2 \left( k-  \sum_{i=1}^k \cos \frac{2 \pi y }{m} \right) } & \leq m e^{-2k} + 2 \sum_{j=1}^{\infty} e^{- \frac{4j^2 \pi^2}{ m^2}  k }.
\end{align*}
The definition of $P_2$ gives 
\begin{align*}
 \sum_{ y \in P_2}  e^{ -2  v \overline t_{n,m} \left(1-  \cos \frac{ \tilde{\beta}\sqrt{ \log m} }{ m} \right) }
& \leq  m^2 e^{ - \overline v \overline{t}_{n,m} \frac{ \overline{\beta} \log m}{ m^2} }   .
\end{align*}
We also have
\begin{align*}
  \quad \sum_{y \in W_I} e^{ -2 e^{-v (\log m)^{1/3}}\overline t_{n,m}\left(1-  \cos  \left(2 \pi e^{-K (\log m)^{2/3}} \right)\right)} \leq
& 
m^n e^{ -\tilde{C} (\log m)^{1/3}\overline t_{n,m} e^{-K (\log m)^{2/3}} }  
\end{align*}
and 
\begin{align*}
 \sum_{y \in Q_I} e^{ -v \overline t_{n,m}\left(1-  \cos  \left(2 \pi/ \tilde{K} \right)\right)} \leq
& 
 m^{I } e^{-K' \overline{t}_{n,m} }
\end{align*}
by the definition of $Q_I$.

Using the fact that $k \geq A^{-1}\overline t_{n,m}$ and putting all the above terms together, we get that there is a constant $B$ such that
\begin{align}\label{mnk}
 &\eqref{30}  \leq  B \frac{1}{n (\log n)^c}.
\end{align}	
Combining \eqref{l} and \eqref{mnk}, we see that there is a universal, positive constant $D$, such that 
\begin{align*}
&\Vert q_{t_{n,m}}-u \Vert_{T.V.}  \leq  D \frac{1}{n (\log n)^c}.
\end{align*}

For $m$ prime, we make a choice of $I $ that allows us to prove a sharper result.
Set $I= \min \{\lfloor \sqrt{ \log m} \rfloor, n-1 \}$, $x= m/8 $ and $w=  m/ \tilde{K}$. 
To prove part (a), we assume that $E_{\overline t_{n,m},x_y}$ is satisfied for a universal constant $A$ that will be determined later in the proof.

\begin{align}
4 \Vert q_{\substack{k, \uw}}-u \Vert_{T.V.}^2  & \leq   \sum_{\substack{y \neq 0 \cr y \in \langle e_1 \rangle}}  e^{ -2 \left( k-  \sum_{i=1}^k \cos \frac{2 \pi y }{m} \right) } +\sum_{ y \in P_2}  e^{ -2    v \overline t_{n,m}\left(1-  \cos \frac{\beta \sqrt{\log m} }{ m} \right) } +   \cr
 & \label{tv} \quad   + \sum_{ y \in  W_I}  e^{ -2  v \overline t_{n,m}\left(1-  \cos \frac{ \pi }{4} \right) } + \sum_{y \in Q_I}  e^{ -2  v \overline t_{n,m}\left( 1-   \cos \frac{2 \pi }{\tilde{K}} \right) }.
\end{align}
Equation \eqref{spl} gives
\begin{align}
 \eqref{tv} & \leq m e^{-2k} +2 \sum_{j=1}^{m/4} e^{- \frac{4j^2 \pi^2}{ m^2}  k } + m^2 e^{ - 2  v \overline t_{n,m} \frac{ \overline{\beta} \log m}{ m} }  \cr
 &+ m^n e^{ -  v \overline t_{n,m}\left(2-  \sqrt{2} \right) }  + m^{I} e^{ -2  v \overline t_{n,m}\left(1-  \cos \frac{2 \pi  \sqrt{\delta} \log m }{m} \right) }
 \cr
 & \leq m e^{-2k} + 2 \sum_{j=1}^{\infty} e^{- \frac{4j^2 \pi^2}{ m^2}  k }
  + m^n e^{ -  v \overline t_{n,m}\left(2-  \sqrt{2} \right) }  
  + m^{I} e^{ - \frac{2  v \overline t_{n,m} \pi^2 }{\tilde{K}^2}  }.
\end{align}
Since $k \geq A^{-1}\overline t_{n,m}$ and choosing $D$ to be a suitable constant, we have that there is a constant $B$ such that
\begin{align}\label{mnk2}
 &\eqref{tv}  \leq  B \frac{1}{n (\log n)^c}.
\end{align}	
Combining \eqref{l} and \eqref{mnk2}, there is a universal, positive constant $D$, such that 
\begin{align*}
& \Vert q_{\overline t_{n,m}}-u \Vert_{T.V.}  \leq  D \frac{1}{n (\log n)^c}.
\end{align*}

For the second part of Proposition \ref{q}, we will only focus on the case of general $m$, since the case $m$ being prime follows the same outline. Lemma \ref{rr1} and a union bound give
\begin{align}
\pr{E^c_{\overline t_{n,m},x_y}} & \leq \pr{ \cup_{y \in W_I} B_{y}^c} + \pr{ \cup_{y \in Q_I} B_{y}^c}+ \pr{ \cup_{y \in P_2} B_{y}^c} \cr
\label{cruc} &\leq m^n b e^{-c} \frac{1}{m^{g n}} + m^I b e^{-c}  \frac{1}{n^{g m}} .
\end{align}
If $I= \lfloor \left( \log m \right)^{1/3} \rfloor \leq n$ then 
\[m^I b e^{-c}  \frac{1}{n^{g m}} = e^{\left( \log m \right)^{4/3} - gm \log n} \leq \frac{1}{n^{(g-1) m}}.\]
If $I=n \leq \lfloor \left( \log m \right)^{1/3}\rfloor$, then
\[m^I b e^{-c}  \frac{1}{n^{g m}} = e^{n \log m  - g m \log n}\leq  \frac{1}{n^{(g-1) m}}.\]
Equation \eqref{cruc} becomes
\[\pr{E^c_{\overline t_{n,m},x_y}}  \leq  b e^{-c} \left(\frac{1}{m^{(g-1) n}} +  \frac{1}{n^{(g-1) m}} \right).\]
\end{proof}
\bibliographystyle{plain}
\bibliography{improvement}

\end{document}